\documentclass[12pt]{article}
	\usepackage{amsmath,fullpage,amsthm,amssymb,xcolor,graphicx}
	\newtheorem{prop}{Proposition}[section]
	\newtheorem{theorem}{Theorem}[section]
	\newtheorem{corollary}{Corollary}[section]
	\newtheorem{lemma}{Lemma}[section]
	\newtheorem{definition}{Definition}[section]

	\DeclareMathOperator{\spec}{spec}

	\title{ Improved bound of graph energy in terms of vertex cover number}
\author{Aniruddha Samanta \thanks{Theoretical Statistics and Mathematics Unit, Indian Statistical Institute, Kolkata-700108, India. Email: aniruddha.sam@gmail.com}}

\date{\today}
\begin{document}
\maketitle
\baselineskip=0.25in

\begin{abstract}
Let $ G $ be a simple graph with the vertex cover number $ \tau $. The energy $ \mathcal{E}(G) $ of $ G $ is the sum of the absolute values of all the adjacency eigenvalues of $ G $. In this article, we establish $ \mathcal{E}(G)\geq 2\tau $ for several classes of graphs. The result significantly improves the known result $ \mathcal{E}(G)\geq 2\tau-2c$ for many classes of graphs, where $ c $ is the number of odd cycles.  
\end{abstract}




\section{Introduction}

Throughout this article, we consider $ G $ to be a simple undirected graph with vertex set $ V(G)=\{ v_1, v_2, \dots, v_n\} $ and edge set $ E(G) $. If two vertices $ v_i $ and $ v_j $ are connected by an edge, then we write $ v_i \sim v_j $ and the edge between them is denoted by $ e_{ij} $. The adjacency matrix $ A(G)=(a_{ij})_{n\times n}$ of $ G $ is an $n\times n $ symmetric matrix, defined as $ a_{ij}=1$ if $ v_i \sim v_j $ and zero otherwise. Let $ \lambda_1, \lambda_2, \dots, \lambda_n$ be the eigenvalues of $ A(G) $. Then the energy of $ G $ is defined as $ \mathcal{E}(G):=\sum\limits_{i=1}^{n}|\lambda_i| $, where $ |\lambda_i| $ is the absolute value of $ \lambda_i $. This graph invariant was formally introduced by Gutman in 1978. It has a great significance in connection with the total $ \pi $ -electron energy in conjugated hydrocarbon in chemistry. Since then, graph energy has been studied extensively by many researchers.

%

Studies of graph eigenvalues have a long history in the mathematics literature. Many beautiful results and bounds have been discovered on the largest and smallest eigenvalues of a graph. However, handling other eigenvalues is difficult, and that results in a very few literature on such eigenvalues. Since, the energy of a graph $ G $ is dependent on all eigenvalues of  $ G $, so it is quite hard to analyses its properties. Therefore, researchers mainly focused on bounding energy in terms of algebraic and combinatorial parameters of a graphs such as matching number, vertex degree, number of vertices, number of edges, vertex cover number, etc.

A vertex cover $ X $ of a graph $ G $ is a subset of $ V(G) $ such that any edge of $ G $ is adjacent to at least one vertex of $ X $. The vertex cover number of $ G $, denoted by $ \tau(G) $, is the cardinality of a vertex cover of $ G $ with minimum number of vertices. A matching of a graph $ G $ is a set of independent edges, that is any two edges have no common vertices. The matching number of $ G $ is the cardinality of a matching with maximum number of edges, and it is denoted by $ \mu(G) $. For a graph $ G $, it is well known that $ \tau(G)\geq \mu(G) $. \\[-2mm]

 Wang and Ma \cite{Wang-Ma}, established the following lower bound of $ \mathcal{E}(G) $ in terms of vertex cover number $ \tau(G)$ and the  number of odd cycles $ c(G) $. 
  
\begin{equation}\label{eq12}
	 \mathcal{E}(G)\geq 2\tau(G)-2c(G).
\end{equation}


For any graph $ G $ with matching number $ \mu(G) $, Wong et.al \cite{Wong-Wang-Chu} proved that 

\begin{equation}\label{eq22}
	\mathcal{E}(G)\geq 2\mu(G).
\end{equation}

Later, many authors extended results \eqref{eq12} and \eqref{eq22} for mixed graphs, digraphs, complex unit gain graphs, etc., see references \cite{Energybound1}.

In this article, we present several classes of graphs $G$ for which $\mathcal{E}(G)\geq 2\tau(G)$, which significantly improves the above bounds \eqref{eq12} and \eqref{eq22} for such classes.

\section{Definitions, notation and preliminary results}\label{prelim}

Let $ G $ be an undirected simple graph with vertex set $ V(G)=\{ v_1, v_2, \dots, v_n\} $ and edge set $ E(G) $. A subset $ S $ of the edge set $ E(G) $ is called a \textit{cut set} of $ G $ if the deletion of all edges of $ S $ from $ G $ increase the number of connected components of $ G $. If $ S $ is a cut set of $ G $ then $ G-S $ denotes the resulting graph after deletion of all edges of $ S $ from $ G $ and it is defined as $ G-S:=G_1\oplus G_2 \oplus \cdots\oplus G_r$, where $ G_i's  $ are the connected components in $ G-S $. For a vertex $ v $ in $ G $, we denote $ G-v $ as an induced subgraph of $ G $ with vertex set $ V(G)\setminus{v} $. A vertex $ v \in V(G) $ is said to be a \textit{cut vertex} of $ G $ if $ G-v $ increases the number of connected components. A \textit{block} of the graph $ G $ is a maximal connected subgraph of $ G $ that has no cut-vertex. For an edge $ e \in E(G) $, we denote $ G-[e] $ as an induce subgraph of $ G $ obtained by removing $ e $, edges adjacent with $ e $ and the vertices joining $ e $. We denote a complete graph and a cycle with $ n $ vertices as $ K_n $ and $ C_n $, respectively. A complete bipartite graph with vertex partition size $p$ and $ q $ is denoted by $ K_{p,q} $. Let us first present the following bound which we are going to improve in this article for several classes of graphs.


\begin{theorem}\cite[Theorem 4.2]{Wang-Ma}
	Let $ G $ be a graph with vertex cover number $ \tau$ and number of odd cycles $ c$. Then $ \mathcal{E}(G)\geq 2\tau-2c$. Equality occurs if and only if $ G $ is the disjoint union of some $ K_{p,p} $, for some $ p $ and isolated vertices.
\end{theorem}

The following two results we use frequently in the later sections. 
\begin{theorem}\cite[Theorem 3.4]{Day-So}\label{Th1}
	If $ G $ is a graph with a simple cut set $ E $, then $ \mathcal{E}(G)\geq \mathcal{E}(G-E) $.
\end{theorem}

\begin{theorem}\cite[Theorem 3.6]{Day-So}\label{Th2}
	If $ E $ is a cut set between complimentary induced subgraphs $ M $ and $ N $ of $ G $.  Suppose the edges of $ E $ form a star, then $ \mathcal{E}(G)>\mathcal{E}(G-E) $.
\end{theorem}

\begin{lemma}\label{lm2.1}
	For any cycle $ C_n $ with $ n $ vertices, \begin{equation*}
		\mathcal{E}(C_n)= \left\{ \begin{array}{ccc}
			4\frac{\cos\frac{\pi}{n}}{\sin \frac{\pi}{n}} & \mbox{for} & n\equiv 0 \mbox{ mod 4}\\
			\frac{4}{\sin \frac{\pi}{n}} & \mbox{for} & n\equiv 2 \mbox{ mod 4}\\
			\frac{2}{\sin \frac{\pi}{2n}} & \mbox{for} & n\equiv 1 \mbox{ mod 2}.
		\end{array}
		\right.
	\end{equation*}
\end{lemma}

\section{Graph energy in terms of the vertex cover number}

We begin this section with some basic class of graphs $ G $ for which $ \mathcal{E}(G)\geq 2\tau(G)$ holds.

\begin{prop}\label{prop3.1}
	\begin{itemize}
		\item[(1)] Let $ G $ be a complete graph. Then $ \mathcal{E}(G)=2\tau(G) $.
		\item[(2)] Let $ G $ be a bipartite graph. Then $ \mathcal{E}(G)\geq 2\tau(G) $.
		\item [(3)] Let $ G $ be a cycle. Then $ \mathcal{E}(G)\geq 2\tau(G) $.
	\end{itemize}
\end{prop}

\begin{proof}
	Part $ (1) $ and $ (2) $ are easy to observe.\\
	$ (3) $ If $ n $ is even, then $ C_n $ is bipartite. Then $ \mathcal{E}(C_n)\geq 2\mu(C_n)=2\tau(C_n) $. If $ n $ is odd. Then $ n\equiv 1 $ mod $ 2 $. That is $ n=2k+1 $, for $ k=1,2, \cdots $. Now $ \frac{\pi}{2n}< \frac{\pi}{4} $, so $ \sin x<x $. Therefore, $ \sin \frac{\pi}{2n}<\frac{\pi}{2n} $. Now by Lemma \ref{lm2.1}, $ \mathcal{E}(C_n)=\frac{2}{\sin \frac{\pi}{2n}} > \frac{2}{\pi/2n}=\frac{4n}{\pi}$. If \begin{equation}\label{eq1}
		\frac{4n}{\pi}\geq 2\tau(C_n),
	\end{equation} where $ \tau(C_n)=\frac{n+1}{2}=k+1 $. Now $\eqref{eq1} $ is true if and only if $ \frac{4(2k+1)}{\pi}\geq 2(k+1) $, i.e., $ k\geq 2 $. For $ k=1 $, $ C_3 $ is complete. Thus the result follows.
\end{proof}

\begin{lemma}
	Let $ u$ be a quasi-pendent vertex of $ G $ and $ \mathcal{E}(G-u)\geq 2\tau(G-u) $. Then $ \mathcal{E}(G)\geq 2\tau(G) $.
\end{lemma}

\begin{proof} 
	Let $v$ be a pendent vertex of $G$ such that  $ u\sim v $. Then any minimum vertex cover, say $ U $ of $ G $ must contain either $ u $ or $ v $. If it contains $ v $, then we can replace $ v $ by $ u $, and it will be still a minimum vertex cover of $ G $. Now we remove vertex $ u $, then $ \tau(G-u) =\tau(G)-1$. Let $ E $ be the edges from $ u $ to other vertices of $ V(G)\setminus\{v\} $. Then $ \mathcal{E}(G)\geq \mathcal{E}(G-E)=\mathcal{E}(G-u)\oplus[e_{u,v}] \geq 2\tau(G-u)+2=2\tau(G)$.	
\end{proof}

\begin{theorem}\label{Th3.1}
	If $ G $ is a tree with vertex cover number $ \tau(G) $, then $ \mathcal{E}(G)\geq 2\tau(G) $. Equality occurs if and only if $ G $ is an edge.
\end{theorem}

\begin{proof}
	If $ \tau(G)=1 $, then $ G $ is some star $ S_n $ and $ \mathcal{S}_n=2\sqrt{n}\geq 2\tau(G) $. Now for any tree $ H $ such that $ \tau(H)<\tau(G) $, $ \mathcal{E}(H)\geq 2\tau(H) $. Suppose $ G $ is a tree with $ \tau(G)\geq2 $. Let $ u $ be a pendent vertex in $ G $ and $ u \sim v $. Then $ \tau(G-v)=\tau(G)-1 $. Consider a cut set $ E=\{ e_{u,w}: w\in N(v)\setminus \{ u\}\} $. Since the shape of $ E $ is a star, so by Theorem \ref{Th2}, and induction hypothesis, $ \mathcal{E}(G)>\mathcal{E}(G-E)=\mathcal{E}(G-v)+\mathcal{E}(e_{uv})=2+2(\tau(G)-1)=2\tau(G) $. If $ G $ is an edge, then $ \mathcal{E}(G)=2\tau(G) $. Suppose $ G $ is not an edge, then by previous observation, $\mathcal{E}(G)>2\tau(G) $.
\end{proof}

\begin{figure}
	\begin{center}
		\includegraphics[scale= 0.69]{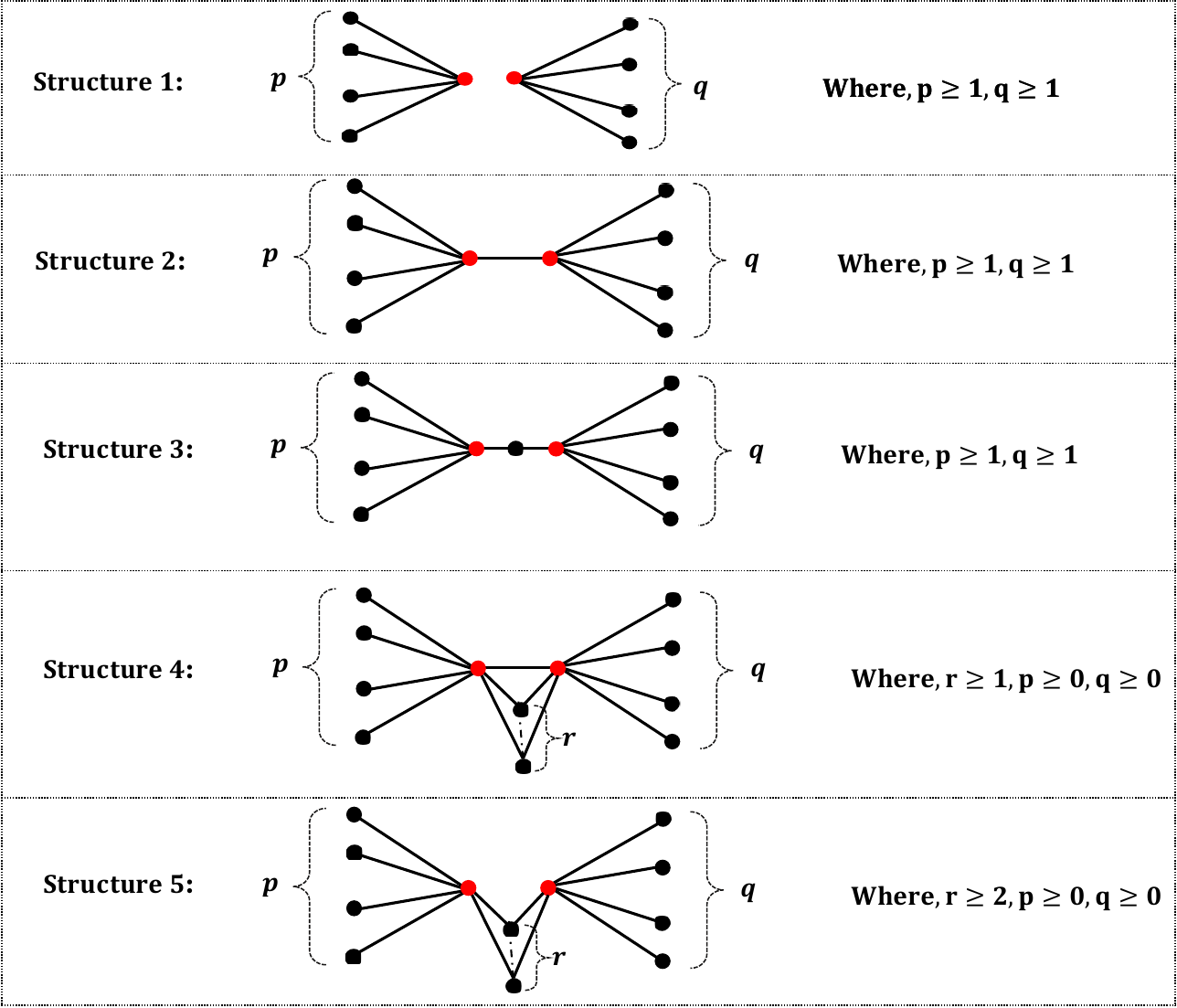}
		\caption{Graphs with vertex cover number 2} \label{fig1}
	\end{center}
\end{figure}

\begin{lemma}
	Let $ G $ be a graph with $ \tau=2 $. Then $ \mathcal{E}(G)\geq 2\tau $.
\end{lemma}

\begin{proof}
	We can observe that a graph with vertex cover number $ \tau=2 $ has any one of the following structures shown in Figure \ref{fig1}. If $ G $ has structure 1, 2 and 3, then $ G $ is bipartite. Therefore, $ \mathcal{E}(G)\geq 2\tau $. Also, for other structures $ G $, $ \mathcal{E}(G)\geq 2\tau $. 
\end{proof}

\begin{figure}
	\begin{center}
		\includegraphics[scale= 0.60]{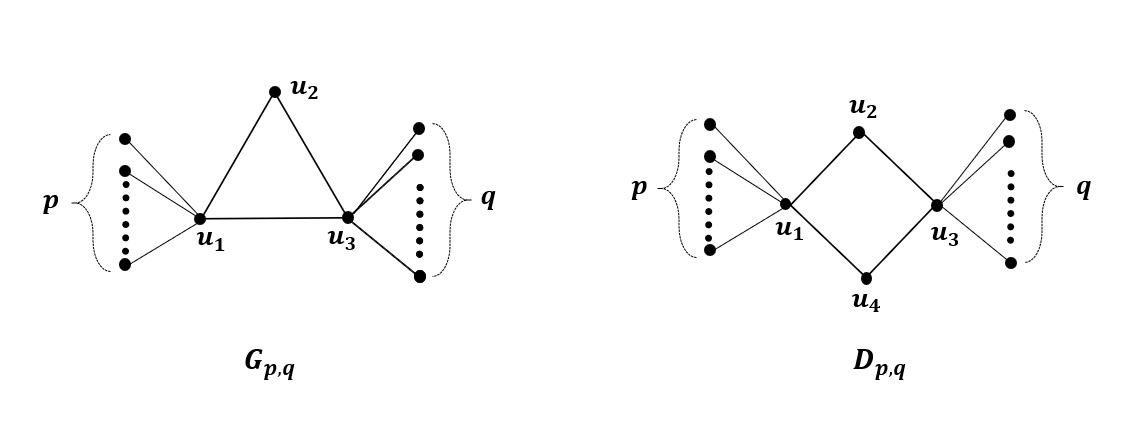}
		\caption{Graphs $ G_{p,q} $ and $ D_{p,q} $} \label{fig2}
	\end{center}
\end{figure}

\begin{lemma}\label{Lm1}
	If $ G $ is either $ G_{p,q} $ or $ D_{p,q} $, where $ p,q \geq 0 $, shown in Figure \ref{fig2}. Then\break $\mathcal{E}(G)\geq 2\tau(G) $.
\end{lemma}
\begin{proof}
	For $ p,q \geq 0 $, $ \tau(G_{p,q})=2=\tau(D_{p,q}) $. If $p=q=0$, then $G_{p,q}=C_3$. Therefore, $ \mathcal{E}(G_{0,0})\geq 2\tau(G_{0,0}) $. Suppose at least one of $ p, q$ is non-zero, say $ p\ne 0 $. Then $\mathcal{E}(G_{p,q})\geq \mathcal{E}([e_{u_2,u_3}])+\mathcal{E}(K_{1,p})=2+2\sqrt{p}\geq 2\tau(G_{p,q})$. Also, for any $ p,q $, $\mathcal{E}(D_{p,q})\geq \mathcal{E}([e_{u_2,u_3}])+\mathcal{E}(K_{1,p+1})=2+2\sqrt{p+1}\geq 2\tau(D_{p,q})$. 
  \end{proof}
%
%

Let us denote $ G-C $ by an induced subgraph of $ G $ with vertices $ V(G)\setminus V(C) $, where $ C $ is a cycle in $ G $. For an edge $ e\in E(G) $, $ G-e $ is obtained from $ G $ by removing $ e $.

\begin{figure}
	\begin{center}
		\includegraphics[scale= 0.50]{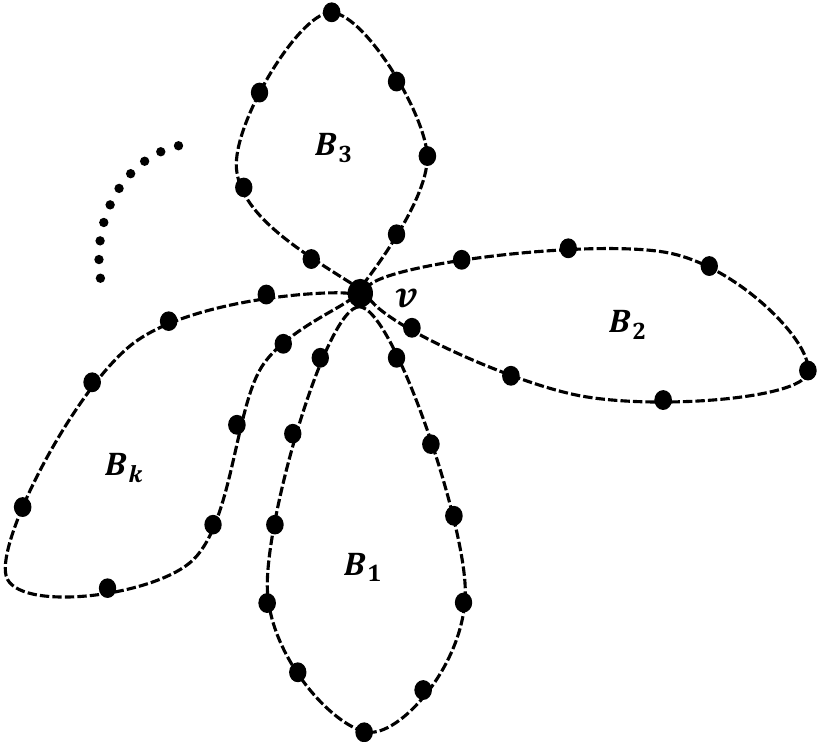}
		\caption{ Graph $ G $} \label{fig3}
	\end{center}
\end{figure}

\begin{lemma}\label{Lm3}
	If $ G $ is a cactus graph with all blocks being cycles and exactly one cut vertex $ v $ (see Figure \ref{fig3}). Then \begin{enumerate}
		\item[(i)]$ \tau(G-v)=\tau(G)-1 $.
		\item[(ii)] $ \mathcal{E}(G)> 2\tau(G) $.
	\end{enumerate}
\end{lemma}
\begin{proof}
	(i) Since the cut vertex $ v $ is of the maximum degree and other vertices have degree $ 2 $, so $ v $ must belong to any minimum vertex cover of $ G $. Therefore, $ \tau(G-v)=\tau(G)-1 $.\\
	(ii) Let $ B_1, B_2, \dots, B_k $ be $ k $ blocks in $ G $, each of which is a cycle, and they have a common vertex $ v $. Since $ G-v=\bigoplus\limits_{i=1}^{k}(B_i-v) $, so $ \tau(G)=1+\sum\limits_{i=1}^{k}\tau(B_i- v) $. Let us consider an induced subgraph $ B_1 $ and its complimentary induced subgraph, say $ S $ in $ G $, where $ S=\bigoplus\limits_{i=2}^{k}(B_i- v) $. Let $ E $ be the cut set such that $ G-E=B_1\oplus S $. Then by Theorem \ref{Th2}, Proposition \ref{prop3.1}(3) and Lemma \ref{Th3.1}, we have $ \mathcal{E}(G)>\mathcal{E}(G-E)\geq 2\{ \tau(B_1)+\sum\limits_{i=2}^{k}\tau(B_i- v)\}=2\tau(G) $.
\end{proof}

A graph $ G $ is called a \textit{cycle-clique} graph if each block of $ G $ is either a cycle or a clique. Some examples of cycle-clique graphs are cactus graphs, friendship graphs, block graphs, graphs with vertex disjoint cycles, trees, etc.

\begin{theorem}\label{Th3.2}
	If $ G $ is a cycle-clique graph. Then $ \mathcal{E}(G)\geq 2\tau(G) $.
\end{theorem}

\begin{proof}
	We prove the result by induction on $ \tau(G) $. If $ \tau(G)=1 $, then $ G \cong S_n $ and $ \mathcal{E}(G)=2\sqrt{n}\geq 2\tau(G) $. If $ \tau(G)=2 $ and $ G $ is a tree then by Theorem \ref{Th3.1}, $ \mathcal{E}(G)\geq 2\tau(G) $. Suppose $ G $ is a cycle-clique graph other than tree with $ \tau(G)=2 $, then either $ G\cong G_{p,q} $ or $ G\cong D_{p,q} $ (see Figure \ref{fig2}). Then by Lemma \ref{Lm1}, $ \mathcal{E}(G)\geq 2\tau(G) $. Let us assume that for any cycle-clique graph $ H $ with $ \tau(H)<\tau(G) $, $ \mathcal{E}(H)\geq 2\tau(H) $. Let $ G $ be any cycle-clique graph with vertex cover number $ \tau(G)>2 $. \\
	\noindent \textbf{Case 1:} If $ G $ has a pendent vertex.\\
	Suppose $ v $ is a pendent vertex and $ u $ is its quasi-pendent vertex. Then we can always find a minimum vertex cover $ U $ of $ G $ such that $ u\in U $.  Minimality of $ |U| $ implies that $ \tau(G- u)=\tau(G)-1 $. Let us take a cut set $ E_1=\{e_{u,w}:w\in N(u)\setminus v \} $. Then by Theorem \ref{Th2} and induction hypothesis, $$ \mathcal{E}(G)>\mathcal{E}(G-E_1)=\mathcal{E}([e_{u,v}])+\mathcal{E}(G-[e_{u,v}])=2+\mathcal{E}(G-u)\geq 2\tau(G).$$\\
	\noindent \textbf{Case 2:} If $ G $ has no pendent vertices. \\
	Let $ B_1, B_2, \dots, B_k $ be the only blocks of $ G $. Then each of $ B_i's $ is either an edge, a cycle, or a complete graph. Construct a tree $ T=(V(T), E(T)) $ obtain by deleting some edges of a graph $ G_1=(V(G_1), E(G_1)) $, where $ V(T)=V(G_1)=\{ B_1, B_2, \dots, B_k\} $. We consider $ B_i\sim B_j $ in $ G_1 $ if and only if $ B_i $ and $ B_j $ has a common vertex. Let $ B_r $ be a block that is either a cycle or a clique containing maximum cut vertices. Take $ B_r $ as a root vertex of $ T $ and put it in level $ 1 $. Then put all vertices of $ N_{G_1}(B_r) $ in level $ 2 $. If the vertices of $ N_{G_1}(B_r) $ are connected by some edges, remove them. Next, for each $ B_i\in N_{G_1}(B_r) $, take its remaining neighbour vertices in level $ 3 $ and remove all edges that arise in level $ 3 $. By the same steps, finally, we get a tree $ T $ (see Figure). Let $ B_{i_1}, B_{i_2}, \dots,  B_{i_p}  $ be the vertices of $ T $ in the top level, say $ i $ and $ B_{i-1} $ be their quasi-pendent vertex.\\
	\noindent \textbf{Case 2.1:} Suppose $ B_{i-1} $ is a cycle or a clique.\\
	It is clear that none of $ B_{i_1}, B_{i_2}, \dots,  B_{i_p} $ are edges in $ G $. Let $ \left(\bigcup\limits_{j=1}^{p} V(B_{i_j})\right) \bigcap V(B_{i-1})=\{u_1, u_2, \dots, u_t\} $. 
	Let us assume that $ u_1 $ is a common vertex of $ B_{i_1}, B_{i_2}, \dots, B_{i_s} $ and $ B_{i-1} $, $ 1\leq s<t $.  Then there is a minimum vertex cover $ U $ of $ G $ such that $ u_1 \in U $. Therefore, $ \tau(G- u_1) =\tau(G)-1$. Let $ G- u_1=S_1\oplus S_2 $, where $ S_1=\bigoplus\limits_{j=1}^{s}(B_{i_j}- u_1) $ and $ S_2 $ is the remaining component. Let $ H=\bigoplus\limits_{j=1}^{s}B_{i_j}$ be an induced subgraph of $ G $. Then $ S_2$ is the complimentary induced subgraph of $ H $ in $ G $. Now $ \tau(H)=\tau(S_1)+1 $, so $ \tau(G)=\tau(H)+\tau(S_2) $. Let $ E $ be a cut set such that $ G-E=H\oplus S_2 $. Therefore, by Theorem \ref{Th2}, $$ \mathcal{E}(G)>\mathcal{E}(G-E)=\mathcal{E}(H)+\mathcal{E}(S_2).$$ Here $ H $ is either a cycle, a complete graph, or graphs of the form given in Figure \ref{fig3}. Then by Proposition \ref{prop3.1}(3) or Lemma \ref{Lm3} and induction hypothesis, $ \mathcal{E}(G)>2\tau(H)+2\tau(S_2)=2\tau(G)$.\\
	\noindent \textbf{Case 2.2:} If $ B_{i-1} $ is an edge. \\
	Let $ B_{i-1}=e_{u,v} $.\\
	\noindent \textbf{Case 2.2.1:} If $ p=1 $. \\
	That is, $ B_{i_1} $ is the only vertex in the top level in $ T $. Then $ B_{i_1}$ is either $C_n $ or $K_n $ and $ u $ is the cut vertex in $B_{i_1} $. It is clear that $ e_{u,v} $ is a cut edge. Then there is a minimum vertex cover $ U $ such that $ u\in U $. Now $ \tau(G- u)=\tau(G)-1$. Also $ G- u=S_1 \oplus S_2 $, where $ S_1=B_{i_1}- u $, $ \tau(S_1)=\tau(B_{i_1})-1 $. Then $ \tau(G)=\tau(S_2)+\tau(B_{i_1}) $. Let $ E=\{ e_{u,v}\} $ be a cut set. Then by Theorem \ref{Th2}, Proposition \ref{prop3.1}(3) and induction hypothesis, $$ \mathcal{E}(G)>\mathcal{E}(G-E)=\mathcal{E}(B_{i_1})+\mathcal{E}(S_2)\geq 2\tau(B_{i_1})+2\tau(S_2)=2\tau(G) $$.\\
	\noindent \textbf{Case 2.2.2:} If $ p\geq 2 $.\\
	Suppose $ B_{i_1}, B_{i_2}, \dots, B_{i_p} $ are in top level in $ T $. Then all $ B_{i_1}, B_{i_2}, \dots, B_{i_p} $ have a common cut vertex say $ u $ with $ B_{i-1} $. Therefore, there exists a minimum vertex cover $ U $ of $ G $ such that $ u\in U $. Then similar to Case $ 2.1.1 $, we can obtain $ \mathcal{E}(G)> 2\tau(G) $.	
\end{proof}

\begin{theorem}\label{th3.3}
	If $ G $ is a connected cycle-clique graph. Then $ \mathcal{E}(G)=2\tau(G) $ if and only if $ G\cong C_4 $ or $ C_3 $ or $ K_n $ or an edge or isolated vertex.
\end{theorem}

\begin{proof}
	Let $ G $ be a cycle-clique graph with $  \mathcal{E}(G)=2\tau(G)$. Suppose $ G $ is not an isolated vertex or an edge. Then we can observe that $ G $ has no pendent vertices, otherwise by Case $ 1 $ in Theorem \ref{Th3.2}, $ \mathcal{E}(G)>2\tau(G) $. Using the way given in Case 2 of Theorem \ref{Th3.2}, we can construct a tree $ T $ from $ G $. It is clear that all pendent vertices in the top-level of $ T $ are cycles or cliques in $ G $. If the quasi-pendent vertex of all top-level vertices of $ T $ is either a cycle or a clique in G, then by Case $ 2.1 $ of Theorem \ref{Th3.2}, $ \mathcal{E}(G)>2\tau(G) $. If the quasi-pendent vertex is an edge in $ G $, then by Case 2.2, $ \mathcal{E}(G)>2\tau(G) $. Therefore, $ T $ has no quasi-pendent vertex. Thus $ T $ has only an isolated vertex. Thus $ G $ has a single block. Therefore, $ G $ is either an edge, a cycle, or a complete graph. If $ G\cong C_n $, for $ n>5 $, then $ \mathcal{E}(C_n)>2\tau(C_n) $. Also for $ G \cong C_3$, $ G\cong C_4 $ or $ G $ is an edge or an isolated vertices, $ \mathcal{E}(G)=2\tau(G) $. If $ G\cong K_n $, then $ \mathcal{E}(G)=2\tau(G) $. 
\end{proof}

\begin{corollary}
	If $ G $ is a block graph, then $ \mathcal{E}(G)\geq 2\tau(G) $. Equality occurs if and only if $ G\cong K_n $, for some $ n\geq 2 $ or isolated vertex.
\end{corollary}

\begin{corollary}
	Let $ G $ be a cactus graph. Then $ \mathcal{E}(G)\geq 2\tau(G) $. Equality occurs if and only if $ G\equiv C_4 $ or $ G\equiv C_3 $ or edge or isolated vertex.
\end{corollary}

\begin{corollary}\label{cor3.3}
			If $ G $ is a connected graph with vertex disjoint cycles. Then $ \mathcal{E}(G)\geq 2\tau(G) $. Equality occurs if and only if $ G\equiv C_4 $ or $ G\equiv C_3 $ or edge or isolated vertex.
\end{corollary}

A graph is called a \textit{split graph} if its vertices can be partitioned into two parts say $ V_1 $  and $ V_2 $ such that one part induces a clique and the other forms an independent set. Any vertices in one part can be adjacent with any vertices in the other part.
	
Since removing isolated vertices from a split graph does not effect the energy and the vertex cover number of the reducing split graph, so it is enough to consider split graphs to be connected. 

\begin{lemma}\label{Lm4}
	Let $ G $ be a split graph such that $ \omega(G)=\tau(G) $. Then there is a set of vertex disjoint induced complete subgraphs $ G_1, G_2, \dots, G_s $ of $ G $ such that $ \tau(G)=\sum\limits_{i=1}^{s}\tau(G_i) $. 
\end{lemma}

\begin{proof}
	Let $ K_p $ be the maximal complete subgraph of $ G $. Then $ V(G)\setminus V(K_p) $ is a vertex independent set of $ G $. Since $ \tau(G)=\omega(G)=p $, so $ V(K_p) $ is a minimum vertex cover of $ G $. Let $\{u_1, u_2, \dots, u_s \}$ be a minimal subset of $ V(G)\setminus V(K_p) $ such that $ \bigcup\limits_{i=1}^{s} N_{G}(u_i)=V(K_p) $. For each $ i=1, \dots, s $, let $ H_i $ be a subgraph induced by the vertices $ N_{G}[u_i] $. Then $ H_i $ is complete. Let $ G_1=H_1 $, and $ G_i=H_i- \bigcup\limits_{j=1}^{i-1}V(H_j) $, $ i=2,\dots, s $. Each $ G_i $ is a complete subgraph of $ G $ and is a vertex disjoint. Then $ \sum\limits_{i=1}^{s}\tau(G_i)= \sum\limits_{i=1}^{s}(|V(G_i)|-1)=|V(K_p)|=\tau(G)$.  
\end{proof}

\begin{theorem}\label{Th3.4}
	If $ G $ is a split graph, then $ \mathcal{E}(G)\geq 2\tau(G) $.
\end{theorem}

\begin{proof}
	Let $ G $ be a split graph with vertex set $ V(G)=\{ v_1, v_2, \dots, v_n\} $. Let $ K_p $ be the maximal complete subgraph of $ G $. So $ \omega(G)=p$. Then $ \tau(G)=p $ or $ p-1 $. Suppose $ \tau(G)=p-1 $. W.l.o.g, consider $ V(K_p)=\{v_1,v_2, \dots,v_p\}$. Then $\{ v_{p+1}, \dots, v_{n}\} $ is a vertex independent set. Let $ E $ be a cut set containing the edges between $ V(K_p) $ and the above vertex-independent set. Then By Theorem \ref{Th1}, $ \mathcal{E}(G)\geq \mathcal{E}(G-E)=\mathcal{E}(K_p)=2\tau(G) $. Suppose $ \tau(G)=p $, Then by Lemma \ref{Lm4}, there is a set of vertex disjoint induced complete subgraphs $ G_1, G_2, \dots, G_s $ such that $ \tau(G)=\sum\limits_{i=1}^{s}\tau(G_i) $. Since $ G_i's $ are vertex disjoint induced complete subgraphs of $ G $, so by Theorem and $ \mathcal{E}(G)\geq \sum\limits_{i=1}^{s}\mathcal{E}(G_i) =2\sum\limits_{i=1}^{s}\tau(G_i)=2\tau(G)$. 
\end{proof}

Some particular type of split graphs are \textit{threshold graphs}, \textit{nested split graph}, \textit{complete split graph etc.} 
\begin{corollary}
	If $ G $ is a threshold graph, then $ \mathcal{E}(G)\geq 2\tau(G) $.
\end{corollary}

A graph obtained by joining a vertex $ u $ to every vertex of a cycle $ C_n $ is known as \emph{wheel graph} and is denoted by $ W_n $. The vertex $ u $ is called the center of $ W_n $. Let $ u_1,u_2, \cdots, u_m $ be $ m $ vertices. A graph obtained by joining each vertex $ u_i $ to every vertex of $ C_n $, for $ i=1,2,\cdots,m $ is denoted by $ W_{m,n} $. Note that $ W_{1,n}=W_n $.

Let $ G_1 $ and $ G_2 $ be two graphs. Then the \emph{join} of $ G_1 $ and $ G_2 $ is another graph induced by joining each vertex of $ G_1 $ to every vertex of $ G_2 $ and is denoted by $ G_1 \vee G_2 $.

\begin{lemma}\label{lm3.5}
If $ G=W_{m,n}$, for some $ m,n\in \mathbb{N} $. Then $ \mathcal{E}(W_{m,n})=\mathcal{E}(C_n)+2\sqrt{mn+1}-2 $.	
\end{lemma}
\begin{proof}
	Let $G_1$ be a graph of $ m $ isolated vertices. Then $C_n \vee G_1=W_{m,n}$. Then $\spec(W_{m,n})=\{ 1-\sqrt{mn+1},1+\sqrt{mn+1}\} \cup \spec(C_n)\setminus \{2\}$ ( See known result). Thus, $\mathcal{E}(W_{m,n})=\mathcal{E}(C_n)+2\sqrt{mn+1}-2.$
\end{proof}

\begin{theorem}
	If $ G=W_{m,n} $, for some $ m,n\in \mathbb{N} $. Then $ \mathcal{E}(G)\geq 2\tau(G) $ and equality occur if and only if $ G\simeq W_{1,3} $ or $ W_{1,4} $.
\end{theorem}

\begin{proof}
	It is clear that $ \tau(G)=\tau(W_{m,n})=\tau(C_n)+m$, if $ 1\leq m\leq \lfloor \frac{n}{2}\rfloor $ and $ n $ otherwise.  Also, for any $ n\geq 3 $, $ \sqrt{mn+1}-1\geq m $. Therefore, by Lemma \ref{lm3.5}, $ \mathcal{E}(G)=\mathcal{E}(C_n)+ 2\sqrt{mn+1}-2\geq 2\tau(C_n)+2\sqrt{mn+1}-2\geq 2\tau(G)$. In fact, equality occur if and only if $ \mathcal{E}(C_n)=2\tau(C_n)$ and $ \sqrt{mn+1}-1=m $. That is, by Corollary \ref{cor3.3}, equality occur if and only if $ G\simeq W_{1,3} $ or $ W_{1,4} $ (See Figure \ref{fig4} ).
\end{proof}

\begin{figure}
	\begin{center}
		\includegraphics[scale= 0.35]{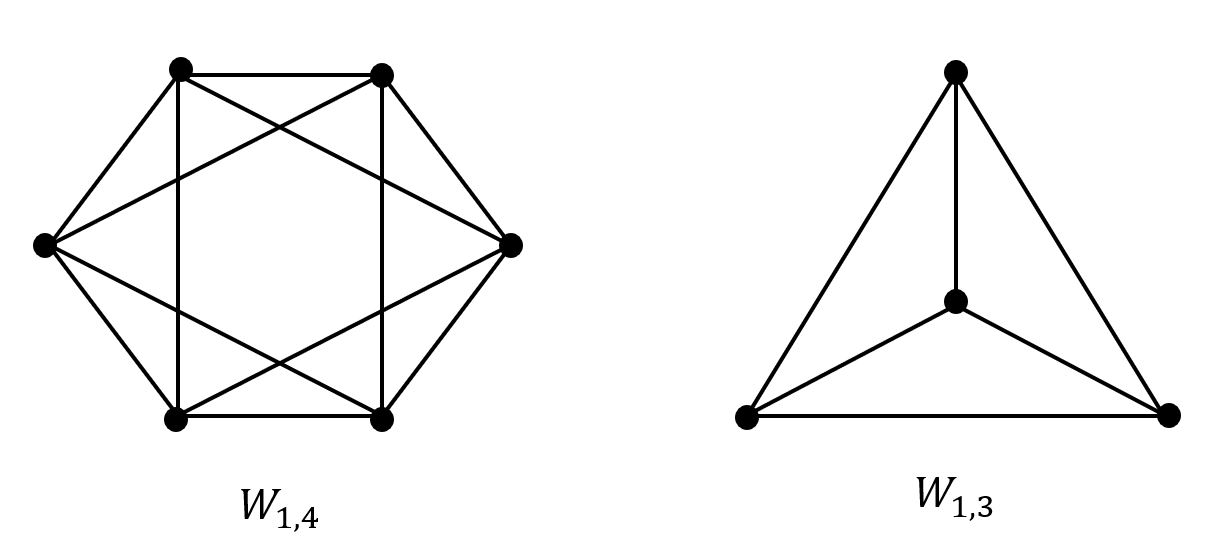}
		\caption{Graphs $W_{1,4}$ and $W_{1,3}$} \label{fig4}
	\end{center}
\end{figure}

\begin{definition}
Let $ G $ be a connected graph with vertex set $ V(G) $ and a minimum vertex cover $ \{ v_1, v_2, \dots, v_\tau\} $. Partition the vertex set $ V(G) $ into two sets $ X:=\{ v_1, \dots, v_\tau\} $ and $ Y:=V(G)\setminus X $. In VC-representation, the graph $ G $ is visualized through $X$ and $Y$, where the vertices of $X$ and $ Y$ form a minimum vertex cover and a vertex-independent set of $ G $, respectively, see Figure \ref{fig5}.
\end{definition}

\begin{figure}[htb!]
	\begin{center}
		\includegraphics[scale= 0.70]{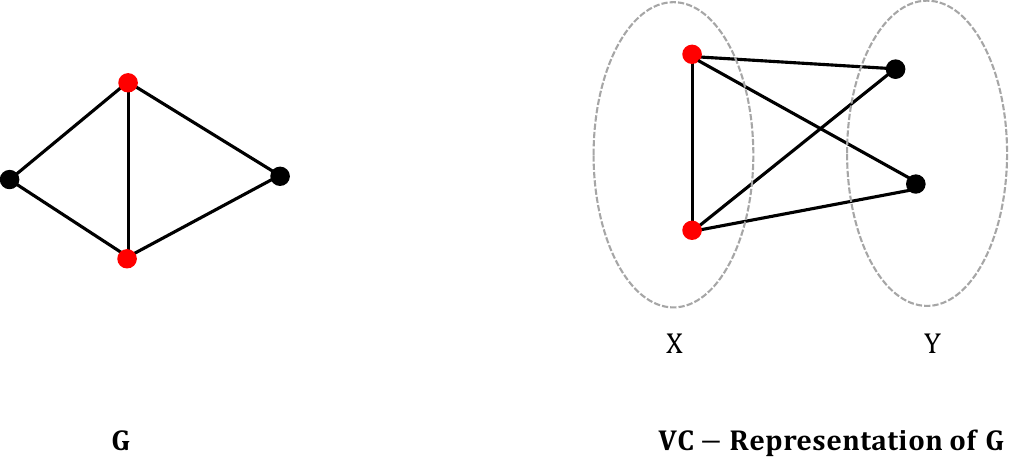}
		\caption{Graph $ G $ and its VC-Representation} \label{fig5}
	\end{center}
\end{figure}

 \begin{definition}
 	Let $ G $ be a connected graph with a minimum vertex cover $ X$. An associated split graph of $ G $, denoted by $ G_s $, is obtained by adding some edges in $ G $ such that the resulting subgraph induced by $ X $ forms a clique.
 \end{definition}
It is to be observed that $ G $ is a subgraph of $ G_s $. In VC-representation of $ G $, every vertex of $ X $ is connected with at least a vertex in $ Y $, and hence the same happens for $ G_s $. 
\begin{prop}\label{Prop3.1}
	Let $ G $ be a connected graph and $ G_s $ be an associated split graph of $ G $. Then $ \tau(G)=\tau(G_s) $.
\end{prop}

\begin{theorem}
	Let $ G $ be a graph such that $ \mathcal{E} (G)\geq \mathcal{E}(G_s)$, for some associated split graph $ G_s $. Then $ \mathcal{E}(G)\geq 2\tau(G) $.
\end{theorem}
\begin{proof}
	The proof follows from Theorem \ref{Th3.4} and Proposition \ref{Prop3.1}. 
\end{proof}

Let $ G $ and $ H $ be two graphs with vertex sets $ V(G) $ and $ V(H) $, respectively. The \textit{Cartesian product } of $ G $ and $ H $ is a graph, denoted by $ G \times H $ with vertex set $ V(G)\times V(H) $ such that $ (g_1,h_1)\sim (g_2, h_2) $ if and only if either (i) $ g_1=g_2 $ and $ h_1\sim h_2 $ or (ii) $ g_1 \sim g_2 $ and $ h_1=h_2 $, where $ (g_i, h_i) \in V(G)\times V(H)$, $ i=1,2 $. 
\begin{prop}\label{prop3.2}
	For any positive integer $ n $, $ \tau(K_n \times K_2)=2\tau(K_n) $.
\end{prop}
\begin{proof}
	It is obvious that $ \tau(K_n\times K_2)\geq 2 \tau(K_n) $. On the other hand, let us assume that $ V(K_n\times K_2)=\{ w_1, w_2, \dots, w_n, w_{n+1}, \dots, w_{2n}\} $. Consider a subset $ W=\{ w_1, w_2, \dots, w_n\} $ be such that $ W $ induces a complete graph $ K_n $ and $ w_n \sim w_{n+1} $. Let us take $ U=V(K_n \times K_2)\setminus\{ w_n, w_{2n}\} $. Then $ U $ forms a vertex cover of $ K_n \times K_2 $. Therefore $ \tau(K_n\times K_2)\leq 2n-2 $. Thus $ \tau(K_n \times K_2)=2\tau(K_n) $. 
\end{proof}

\begin{theorem}
	Let $ G=K_n\times K_2 $ be a graph, where $ n $ is a positive integer. Then $ \mathcal{E}(G)=2\tau(G) $.
\end{theorem}
\begin{proof}
	By \cite[Lemma 3.26]{Bapat} and Proposition \ref{prop3.2}, we have $ \mathcal{E}(G)=4(n-1)=2\tau(G) $.
\end{proof}

\section{Conclusion}
In this article, we establish the bound $ \mathcal{E}(G)\geq 2\tau $ for the following class of graphs. Cycles, Bipartite graphs, Complete graphs, cycle-clique graphs (some examples: cactus graphs, friendship graphs, block graphs, graphs with vertex disjoint cycles), Split graphs (some examples: threshold graphs, nested split graphs, complete split graphs), wheel graphs, $ W_{m,n} $ (defined earlier), some graphs obtained by cartesian product and join of graphs. Further we discuss equality of the bound for some class of graphs. 

 \section*{Acknowledgments}
	Aniruddha Samanta expresses thanks to the National Board for Higher Mathematics (NBHM), Department of Atomic Energy, India, for providing financial support in the form of an NBHM Post-doctoral Fellowship (Sanction Order No. 0204/21/2023/R\&D-II/10038). The author also acknowledges excellent working conditions in the Theoretical Statistics and Mathematics Unit, Indian Statistical Institute Kolkata. 


    \bibliographystyle{amsplain}

\begin{thebibliography}{99}\baselineskip18pt


\bibitem{Bapat}
R. B. Bapat, \emph{Graphs and matrices}, Universitext, Springer, London; Hindustan Book Agency, New Delhi, 2010.

\bibitem{Day-So}
ane Day and Wasin So, \emph{Graph energy change due to edge deletion}, Linear Algebra Appl.
\textbf{428} (2008), no. 8-9, 2070–2078.


\bibitem{Energybound1}
Aniruddha Samanta and M Rajesh Kannan, \emph{Bounds for the energy of a complex unit
gain graph}, Linear Algebra and its Appl. \textbf{612} (2021), 1–29.

\bibitem{Wang-Ma}
Long Wang and Xiaobin Ma, \emph{Bounds of graph energy in terms of vertex cover number},
Linear Algebra Appl. \textbf{517} (2017), 207–216. 



\bibitem{Wong-Wang-Chu}
Dein Wong, Xinlei Wang, and Rui Chu, \emph{Lower bounds of graph energy in terms of matching number}, Linear Algebra Appl. \textbf{549} (2018), 276–286. 
\end{thebibliography}

\mbox{}

\end{document}